\documentclass[11pt]{article}
\usepackage[english,activeacute]{babel}
\usepackage{amsmath,amsfonts,amssymb,amstext,amsthm,amscd,mathrsfs,amsbsy,xypic,centernot}
\usepackage{xcolor}
\usepackage{MnSymbol}
\usepackage{graphicx}
\usepackage[all]{xy}
\usepackage{hyperref}

\newtheorem{teo*}{Theorem}
\newtheorem{teo}{Theorem}[section]
\newtheorem{pro}[teo]{Proposition}

\newtheorem{lem}[teo]{Lemma}

\theoremstyle{definition}
\newtheorem{defi}[teo]{Definition}
\newtheorem{exam}[teo]{Example}
\newtheorem{rem}[teo]{Remark}

\newtheorem{nota}[teo]{Notation}

\newcommand{\N}{\mathbb N}

\newcommand{\K}{\mathbb K}

\newcommand{\Der}{\operatorname{Der}}

\newcommand{\Hom}{\operatorname{Hom}}

\providecommand{\keywords}[1]{{\textbf{Keywords:}} #1}

\begin{document}

\title{High-order derivations of the Hasse-Schmidt algebra}
\author{Paul Barajas}
\maketitle

\begin{abstract}
In this paper we establish relations among the module of high-order derivations of the Hasse-Schmidt algebra and the module of high-order derivations of the base ring.
\end{abstract}

\noindent\keywords{Hasse-Schmidt derivation, Hasse-Schmidt algebra, high-order derivation}

\tableofcontents
\newpage


\section*{Introduction}

The notion of derivation and its generalizations have been very important objects of study within commutative algebra and, in turn, powerful tools in areas such as algebraic geometry and differential geometry. Since there are different notions of high-order derivations, it is natural to ask if there are relations among them (see \cite{N}, \cite{LN1}, \cite{LN2}, \cite{LN3}).

In a recent work, T. de Fernex and R. Docampo showed how derivations of the Hasse-Schmidt algebra and derivations of the base ring are related. The statement is the following:
\begin{teo*}\cite{dFDo1}\label{Der}
Let $n\in\N\cup\{\infty\}.$ Let $k$ be a ring, $A$ be a $k$-algebra and $HS^n_{A/k}$ be the Hasse-Schmidt algebra of $A$ over $k$ of order $n$. Let $M$ be a $HS_{A/k}^n$-module. There is a natural isomorphism of $HS_{A/k}^n$-modules
\[\Der_k(HS_{A/k}^n,M)\simeq\Der_k(A,M\otimes_{HS_{A/k}^n}HS_{A/k}^n[t]/\langle t^{n+1}\rangle).\]
\end{teo*}

This result has important implications for jet schemes and arc spaces (see \cite{dFDo1}). It also has applications in the theory of Nash blowups (see \cite{dFDo2}). In addition, Theorem \ref{Der} implies an isomorphism on the dual objects, that is, the modules of K\"ahler differentials. This dual version has also been studied by Chiu and Narváez using the Hasse-Schmidt differential module as the main tool (see \cite{CN}).
 
In this work, we continue to explore relations between different generalizations of a derivation. H. Osborn and Y. Nakai introduced and studied high-order derivations (\cite{O,N}). The high-order derivation defined by Osborn and Nakai are related to the differential operators defined by Grothendieck in \cite{EGA}. Specifically, a high-order derivation is a differential operator with a zero constant term. Our main theorem relates the module of high-order derivations of $HS^n_{A/k}$ with the module of high-order derivations of $A$. This result can be seen as a higher-order version of Theorem \ref{Der}.

\begin{teo*}\label{mT*}
Let $n\in\N\cup\{\infty\}.$ Let $m\in\N_{\geq1}$. Let $A$ be a $k$-algebra and let $M$ be a $HS_{A/k}^n$-module.
\begin{enumerate}
\item[(1)] There is a natural homomorphism of $HS_{A/k}^n$-modules
$$\phi_n:\Der_k^m(HS_{A/k}^n,M)\to\Der_k^m(A,M\otimes_{HS_{A/k}^n}HS_{A/k}^n[t]/\langle t^{n+1}\rangle).$$ 
\item[(2)]  Let $\K$ be a field of characteristic zero. Assume that $A$ is a finitely generated $\K$-algebra. Then $\phi_n$ is surjective.
\item[(3)] In general, $\phi_n$ is not injective.
\end{enumerate}
\end{teo*}

Contrary to the case of order $m=1$, we see that there is no isomorphism in the case of $m>1$. In the former case, the notion of trivial extension of a module is an important tool to prove the isomorphism. It is likely that the non-isomorphism of the latter case is related to the lack of a notion of an $m$-trivial extension that is suitable for $m$-order derivations.

In recent years, many authors have studied high-order versions of classical results in the theory of differentials. For instance, a higher-order version of a Jacobian matrix was introduced to study the module of high-order K\"ahler differentials of finitely generated algebras (see \cite{DD2}). Regularity of certain rings was characterized using the module of principal parts (see \cite{BD1, BJNB}). A characterization of the $k$-torsion freeness of the module of high-order differentials was given in terms of a singular locus (see \cite{DH}). New numerical invariants of rings related to the rank of the module of principal parts were introduced (see \cite{BJNB}). Moreover, these algebraic results have had geometric applications: Nobile-like theorems for the higher Nash blow-up (see \cite{DD1, DDLN2, DDLN1,BD}), the study of higher Nash blowups of toric varieties (see \cite{ADE,EC}), and the introduction of new invariants of complex analytic hypersurfaces (see \cite{HMY}). The present article can be placed in this series of recent works. 

Let us describe the content of this paper. In Section 1, we recall the concepts of Hasse-Schmidt derivations and high-order derivations, and present some of its basic properties. In Section 2 we prove Theorem \ref{Der}. Section 3 study the implications of our main theorem on the module of high-order differential. Finally, in Section 4 we present a formula that relates the usual partial derivatives of high-order and the universal Hasse-Schmidt derivations.

\textit{Throughout the paper, we always consider commutative rings with unity.}

\section{High-order derivations and the Hasse-Schmidt algebra}
\subsection{The Hasse-Schmidt derivations and the Hasse-Schmidt algebra}

In this section, we recall the definition of Hasse-Schmidt derivations and some of its basic facts, following \cite{V}. We also present the theorem by T. de Fernex and R. Docampo that relates derivations of the Hasse-Schmidt algebra with derivations of the base ring.

\begin{defi}\cite[Section 1]{V}
Let $n\in\N\cup\{\infty\}.$ Let $k$ be a ring. Let $A$, $B$ be $k$-algebras and let $f:k\to A$ be the structural morphism. A Hasse-Schmidt derivation of order $n$ from $A$ to $B$ over $k$ is a sequence $(D_0,\ldots,D_n) $ (or $(D_0,D_1,\ldots)$ if $n=\infty$), where $D_0:A\to B$ is a $k$-algebra homomorphism, $D_i:A\to B$ are $k$-linear maps, $D_i(f(c))=0$ for $i\in\{1\ldots,n\}$ and $c\in k$, and for all $x,y\in A$ and all $l\in\{0,\ldots, n\}$, the maps $D_l$ satisfy the following rule: 
\[D_l(xy)=\sum_{i+j=l}D_i(x)D_j(y).\]  
\end{defi}

\begin{defi}\cite[Section 1]{V}
The Hasse-Schmidt algebra $HS_{A/k}^n$ is defined as the quotient of a polynomial algebra
\[HS_{A/k}^n=A[x^{(i)}| x\in A, i\in\{1,\ldots,n\}]/I,\] where $I$ is the ideal generated by $\{(x+y)^{(i)}-x^{(i)}-y^{(i)}|x,y\in A, i\in\{1,\ldots,n\}\},$ $\{f(c)^{(i)}|c\in k, i\in\{1,\ldots,n\}\},\{(xy)^{(i)}-\sum_{j+k=i}x^{(j)}y^{(k)}| x,y\in A, i\in\{0,\ldots,n\}\},$ where we identify $x^{(0)}$ with $x$ for all $x\in A.$ We define the universal Hasse-Schmidt derivation $(d_0, d_1,\ldots,d_n)$ from $A$ to $HS_{A/k}^n$ by $d_i(x)=x^{(i)}+I.$
\end{defi}

\begin{nota}\label{nota}
Following the notation of \cite{dFDo1,CN}, we denote by $A_n=HS_{A/k}^n$, $B_n=\frac{A_n[t]}{\langle t^{n+1}\rangle}$ if $n$ is finite, and $B_{\infty}=A_{\infty}[[t]]$ if $n=\infty$. Consider the $k$-algebra homomorphism $\gamma_n^{\#}:A\to B_n,$ $a\mapsto \sum_{j=0}^n d_j(a)t^j.$ Note that $B_n$ is an $A_n$-module and it is also an $A$-module via $\gamma_n^{\#}.$ For $n\in\N\cup\{\infty\},$ define the $A_n$-module $B_n^{\infty}=A_{\infty}[t]/\langle t^{n+1}\rangle.$
\end{nota}

\begin{rem}\label{Blim}
Notice that $B_{\infty}\simeq\underset{n\in\N}{\varprojlim}B_n^{\infty}$ (see \cite[p.~8]{CN}).
\end{rem}

The following result was the main motivation for this paper.

\begin{teo}\cite[Lemma 5.1]{dFDo1}\label{Der1}
Let $n\in\N\cup\{\infty\}.$ Let $k$ be a ring and $A$ be an $k$-algebra. Let $M$ be an $A_n$-module, and consider $M\otimes_{A_n} B_n$ with the $A$-module structure induced from the $A$-module structure on $B_n.$ Then there is a natural isomorphism of $A_n$-modules
\[\Der_k(A_n,M)\simeq\Der_k(A,M\otimes_{A_n}B_n).\]
\end{teo}
%
\subsection{Main theorem}

In this section we state our main theorem. Let us first recall the definition of derivations of order $m\geq1$ and some of its basic properties.

\begin{defi}\cite[Chapter I-1]{N}\label{def der}
	Let $k$ be a ring. Let $A$ be a $k$-algebra and $M$ be an $A$-modulo. A $k$-derivation of order $m\in\N_{\geq1}$ is a $k$-homomorphism $\Delta$ from $A$ to $M$ that satisfy the following identity: 
\begin{align}\label{ruleN}
\Delta(x_0\cdots x_m)= \sum_{s=1}^{m}(-1)^{s-1}\sum_{i_1< \dots <i_s}x_{i_1}\cdots x_{i_s}\Delta(x_0\cdots \check{x}_{i_1}\cdots \check{x}_{i_s}\cdots x_m),
\end{align}
for any set $\{x_0,\dots, x_m\}$ of $m+1$ elements of $A$. The symbol $\check{x}_{i_j}$ means that this element does not appear in the product.  Notice that a $k$-derivation of order $1$ is a usual derivation. We denote by $\Der_k^m(A,M)$ the module of derivations of order $m$ from $A$ to $M$ over $k$. 
\end{defi}

A basic relation among Hasse-Schmidt derivations and high-order derivations is given in the following proposition.

\begin{pro}{\cite[Proposition 5]{N}}
Let $n,m\in\N$. Let $D=(D_0,D_1,\ldots,D_n)$ be a Hasse-Schmidt derivation of order $n$. Then the $m$-th component $D_m$ is a $k$-derivation of order $m$ for $1\leq m \leq n.$
\end{pro}

Let $A$ be a $k$-algebra and denote $I_A:=\ker(A\otimes_{k}A\rightarrow A$, $a\otimes b\mapsto ab)$. Give structure of $A$-module to $A\otimes_kA$ by multiplying on the left entry. For $m\in\N_{\geq1}$, define the $A$-module $$\Omega^{(m)}_{A/k}:=I_A/I_A^{m+1}.$$  Define the map $d_A^m:A\to\Omega^{(m)}_{A/k},$ $a\mapsto 1\otimes a-a\otimes 1 + I_A^{m+1}.$

\begin{defi}\cite[Chapter II-1]{N}
The $A$-module $\Omega^{(m)}_{A/k}$ is called the module of K\"ahler differentials of order $m$ of $A$ over $k$. The map $d_A^m$ is a derivation of order $m$ and is called the canonical derivation of order $m$ of $A$.
\end{defi}

\begin{teo}\cite[Proposition 1.6]{O}\label{upd}
Let $D: A \to M$ be a $k$-derivation of order $m$. There is a unique homomorphism $g:\Omega_{A/k}^{(m)}\to M$ of $A$-modules such that $D=g\circ d_A^m.$
\end{teo}

\begin{rem}\label{Omelim}
For $0\leq i\leq j\leq\infty$, by the previous universal property, $d_{A_j}^m|_{A_i}:A_i\to\Omega_{A_j}^{(m)}$ induces an $A$-homomorphism $g_{i,j}:\Omega_{A_i}^{(m)}\to\Omega_{A_j}^{(m)}.$ These homomorphisms satisfy \[g_{j,k}\circ g_{i,j}=g_{i,k}, \mbox{   }\textnormal{ for all } 0\leq i\leq j\leq k\leq\infty\] \[g_{i,i}=Id_{A_i}, \mbox{   }\textnormal{ for all } i\in\N\cup\{\infty\}.\] Hence, they induce a direct system, and \begin{align}\label{Omeinf}\Omega_{A_{\infty}}^{(m)}=\underset{n\in\N}{\varinjlim}\Omega_{A_{n}}^{(m)}.\end{align}
\end{rem}

We can now state our main theorem.

\begin{teo}\label{mT}
Let $n\in\N\cup\{\infty\}.$ Let $m\in\N_{\geq1}$. Let $A$ be a $k$-algebra. Let $M$ be a $A_n$-module, and consider $M\otimes_{A_n} B_n$ with the $A$-module structure induced from the $A$-module structure on $B_n.$ 
\begin{enumerate}
\item[(1)] There is a natural homomorphism of $A_n$-modules
$$\phi_n:\Der_k^m(A_n,M)\to\Der_k^m(A,M\otimes_{A_n}B_n).$$
\item[(2)] Let $\K$ be a field of characteristic zero. Assume that $A$ is a finitely generated $\K$-algebra. Then $\phi_n$ is surjective.
\item[(3)] In general, $\phi_n$ is not injective.
\end{enumerate}
\end{teo}

In all the discussion that follows, the number $m$ will always be fixed. For this reason we do not include the letter $m$ in the notation $\phi_n.$
\section{Proof of Theorem \ref{mT}}

We use notation \ref{nota} throughout this section.

\subsection{Proof of Theorem \ref{mT} $(1)$}

\begin{lem}\label{An-A}
Let $\rho_n:\Der_k^m(A,M\otimes_{A_n}B_n)\to\Hom_{A}(\Omega_{A/k}^{(m)},M\otimes_{A_n}B_n)$ be the isomorphism of $A$-modules of Theorem \ref{upd}. Then $\rho_n$ is also an isomorphism of $A_n$-modules.
\end{lem}
\begin{proof}
Since $M\otimes_{A_n}B_n$ is a $A_n$-module, we can give structure of $A_n$-module to $\Der^m_k(A,M\otimes_{A_n} B_n)$ via the function 
\[\begin{array}{rccl}
&A_n\times \Der^m_k(A,M\otimes_{A_n} B_n)&\longrightarrow&\Der^m_k(A,M\otimes_{A_n} B_n)\\
&(F,D)&\mapsto&FD,
\end{array}\]
where $FD: A\to M\otimes_{A_n} B_n$, $c\mapsto FD(c)$ (this notation means to multiply $F$ on the second entry). In the same way we give structure of $A_n$-module to $\Hom_{A}(\Omega_{A/k}^{(m)},M\otimes_{A_n}B_n)$. With these structures $\rho_n$ is a $A_n$-homomorphism. Indeed, let $F\in A_n$ and $D\in\Der^m_k(A,M\otimes_{A_n} B_n).$ For $d_A^m(f)\in\Omega_{A/k}^{(m)}$ we have
\begin{align}
\Big(\rho_n(FD)\Big)(d_A^m(f))&=FD(f)=F \Big(D(f)\Big)=F\Big(\rho_n(D)(d_A^m(f))\Big).\notag
\end{align}
Hence, $\rho_n(FD)=F\rho_n(D).$
\end{proof}

\begin{pro}\label{der m}
Let $n\in\N$, $m\in\N_{\geq1}$. Let $M$ be an $A_n$-module. Consider $M\otimes_{A_n}B_n$ with the $A$-module structure induced by that of $B_n$.
Let $D\in \Der^m_k(A_n,M)$ and consider the map
\begin{align}
\overline{D}:A&\rightarrow M\otimes_{A_n}B_n\notag\\
a&\mapsto \sum_{i=0}^n D\circ d_i(a)\otimes t^i.\notag
\end{align}
Then $\overline{D}\in\Der^m_k(A,M\otimes_{A_n} B_n)$. In particular, we have a well-defined homomorphism of $A_n$-modules 
$\phi_n:\Der^m_k(A_n,M)\rightarrow \Der^m_k(A,M\otimes_{A_n} B_n)$, $D\mapsto \overline{D}$.
\end{pro}
\begin{proof}
First, by Lemma \ref{An-A}, $\Der^m_k(A,M\otimes_{A_n} B_n)$ has structure of $A_n$-modulo, in particular the following map is $A_n$-linear, 
\[\begin{array}{rccl}
&\Der^m_k(A_n,M)&\longrightarrow&\Der^m_k(A,M\otimes_{A_n} B_n)\\
&D&\mapsto&\overline{D}.
\end{array}\]
Since $D$ and $d_i$ are $k$-linear, the same goes for $\overline{D}$. Now we show that the identity of definition \ref{def der} holds for $\overline{D}$. By definition,
\begin{align}\label{eq1wd} 
&\overline{D}(a_0\cdots a_m)=\sum_{i=0}^nD((a_0\cdots a_m)^{(i)})\otimes t^i=\sum_{i=0}^nD(\sum_{\substack{\alpha\in\mathbb{N}^{m+1}\\|\alpha|=i}} a_0^{(\alpha_0)}\cdots a_m^{(\alpha_{m})})\otimes t^i\notag\\&%
= \sum_{i=0}^{n}\sum_{\substack{\alpha\in\mathbb{N}^{m+1}\\|\alpha|=i}}\sum_{s=1}^{m}(-1)^{s-1}\sum_{j_1< \dots <j_s}a_{j_1}^{(\alpha_{j_1})}\cdots a_{j_s}^{(\alpha_{j_s})}D\Big(a_0^{(\alpha_{0})}\cdots \check{a}_{i_1}^{(\alpha_{j_1})}\cdots \check{a}_{i_s}^{(\alpha_{j_s})}\cdots a_m^{(\alpha_{m})}\Big)\otimes t^i.
\end{align}
For each subset $j=\{j_1,j_2,\ldots,j_s\}\subset\{0,1,\ldots,m\}$ such that $j_1<j_2<\ldots<j_s$, we set 
 \[\begin{split}
D\big(\check{a_j}^{(\alpha)}\big)&:=D\big(a_0^{(\alpha_0)}\cdots \check{a}_{j_1}^{(\alpha_{j_1})}\cdots \check{a}_{j_s}^{(\alpha_{j_s})}\cdots a_m^{(\alpha_{m})}\big),\\
D\big(\check a_j\big)&:=D\big(a_0\cdots\check a_{j_1}\cdots \check a_{j_s}\cdots a_m\big),\\
 a_j^{(\alpha)}&:=a_{j_1}^{\alpha_{j_1}}a_{j_2}^{\alpha_{j_2}}\cdots a_{j_s}^{\alpha_{j_s}},
\end{split}\] where the symbol $\check{a}_{j_i}^{\alpha_{j_i}}$ means that this element does not appear in the product. Using this notation the equation (\ref{eq1wd}) is rewrite as

\begin{align}
\overline{D}(a_0\cdots a_m)=\sum_{i=0}^{n}\big(\sum_{\substack{\alpha\in\mathbb{N}^{m+1}\\|\alpha|=i}}\sum_{s=1}^m(-1)^{s-1} \sum_{j_1<\dots<j_s}a_j^{(\alpha)}D(\check{a_j}^{\alpha})\big)\otimes t^i.
\end{align}
For each $\alpha_j=(\alpha_{j_1},...,\alpha_{j_s})\in\N^s$ we set $\Gamma_j=\{\gamma\in\N^{m+1}|\gamma_{j_k}=\alpha_{j_k}, k\in\{1,\dots,s\}, |\gamma|=i\}$. Notice that
 \begin{align*}
 &\sum_{i=0}^{n}\sum_{s=1}^m(-1)^{s-1} \sum_{j_1<\dots<j_s}\sum_{\substack{\alpha_j\in\N^{s}\\0\leq |\alpha_j|\leq i}}a_j^{(\alpha_j)}\sum_{\gamma\in\Gamma_j}D(\check{a_j}^{(\gamma)})\\
 &=\sum_{i=0}^{n}\sum_{\substack{\alpha\in\mathbb{N}^{m+1}\\|\alpha|=i}}\sum_{s=1}^m(-1)^{s-1} \sum_{j_1<\dots<j_s}a_j^{(\alpha)}D(\check{a_j}^{(\alpha)}).
 \end{align*}
Then 
\[\begin{split}\overline{D}(a_0\cdots a_m)&=\sum_{i=0}^{n}\sum_{s=1}^m(-1)^{s-1} \sum_{j_1<\dots<j_s}\sum_{\substack{\alpha_j\in\N^{s}\\0\leq |\alpha_j|\leq i}}a_j^{(\alpha_j)}\sum_{\gamma\in\Gamma_j}D(\check{a_j}^{(\gamma)})\\&
=\sum_{i=0}^{n}\sum_{s=1}^m(-1)^{s-1} \sum_{j_1<\dots<j_s}\sum_{\substack{\alpha_j\in\N^{s}\\0\leq |\alpha_j|\leq i}}a_j^{(\alpha_j)}D(\sum_{\gamma\in\Gamma_j}\check{a_j}^{(\gamma)}).\end{split}\]
Set $\ell_j=\sum_{k=1}^s\alpha_{j_k}=|\alpha_j|.$ Thus
\[\begin{split}
\sum_{\gamma\in\Gamma_j}\check{a_j}^{(\gamma)}&=\sum_{\gamma\in\Gamma_j}a_0^{(\gamma_0)}\cdots\check{a}_{j_1}^{(\alpha_{j_1})}\cdots\check{a}_{j_s}^{(\alpha_{j_s})}\cdots \check{a}_{m}^{(\gamma_{m})}=\sum_{\substack{\gamma'\in\N^{m+1-s}\\|\gamma'|=i-\ell_j}}a_0^{(\gamma'_0)}a_1^{(\gamma'_1)}\cdots \check{a}_{j_1}\cdots\check{a}_{j_s} \cdots a_m^{(\gamma'_{m+q-s})}\\&=(\check{a_j})^{(i-\ell_j)}.
\end{split}\]
As a consequence,
 \[\begin{split}
\overline{D}(a_0\cdots a_m)&=\sum_{i=0}^{n}\sum_{s=1}^m(-1)^{s-1} \sum_{j_1<\dots<j_s}\sum_{\substack{\alpha_j\in\N^{s}\\0\leq |\alpha_j|\leq i}}a_j^{(\alpha_j)}D\big((\check{a_j})^{(i-\ell_j)}\big)\otimes t^i\\&
=\sum_{i=0}^{n}\sum_{s=1}^m(-1)^{s-1} \sum_{j_1<\dots<j_s}\sum_{\substack{\alpha_j\in\N^{s}\\0\leq |\alpha_j|\leq i}}D\big((\check{a_j})^{(i-\ell_j)}\big)\Big)\otimes a_j^{(\alpha_j)}  t^{\ell_j+i-\ell_j}\\&
=\sum_{i=0}^{n}\sum_{s=1}^m(-1)^{s-1} \sum_{j_1<\dots<j_s}\sum_{\substack{\alpha_j\in\N^{s}\\0\leq |\alpha_j|\leq i}}(a_j^{(\alpha_j)}t^{\ell_j})\Big(D\big((\check{a_j})^{(i-\ell_j)}\big)\otimes   t^{i-\ell_j}\Big).
\end{split}\]
We also have,
\[\begin{split}&\sum_{i=0}^{n}\sum_{\substack{\alpha_j\in\N^{s}\\0\leq |\alpha_j|\leq i}}(a_j^{(\alpha_j)}t^{\ell_j})\Big(D\big((\check{a_j})^{(i-\ell_j)}\big)\otimes   t^{i-\ell_j}\Big)\\&=\sum_{\substack{\alpha_j\in\N^{s}\\0\leq |\alpha_j|\leq i}}\big(\prod_{k=1}^sa_{j_k}^{\alpha_{j_k}}  t^{\ell_j}\big)\sum_{i-\ell_j=0}^n\big(D((\check{a_j})^{(i-\ell_j)}\otimes t^{i-\ell_j}\big).\end{split}\]
Hence,
\[
\begin{split}
\overline{D}(a_0\cdots a_m)=\sum_{s=1}^m(-1)^{s-1}\sum_{j_1<\dots<j_s}\big(\sum_{\substack{\alpha_j\in\N^{s}\\0\leq |\alpha_j|\leq i}}\prod_{k=1}^sa_{j_k}^{\alpha_{j_k}}  t^{\ell_j}\big)\sum^n_{i-\ell_j=0}\big(D((\check{a_j})^{(i-\ell_j)}\otimes  t^{i-\ell_j}\big).
\end{split}\]
Recall that,
\[\prod_{k=1}^s\gamma_n^{\#}(a_{j_k})=\prod_{k=1}^{s}\big(\sum_{i=0}^n(a_{j_k})^{(i)}t^i\big)=\sum_{i=0}^n\big(\sum_{i_1+\cdots+i_s=i}a_{j_1}^{(i_1)}\cdots a_{j_s}^{(i_s)}t^i\big).\]
We conclude,
\[\begin{split}\overline{D}(a_0\cdots a_m)=&\sum_{s=1}^m(-1)^{s-1}\sum_{j_1<\dots<j_s}\big(\prod_{k=1}^s\gamma_n^{\#}(a_{j_k})\big)\big(\overline{D}(a_0\cdots\check a_{j_1}\cdots \check a_{j_s}\cdots a_m\big)\big)\\&
=\sum_{s=1}^m(-1)^{s-1}\sum_{j_1<\dots<j_s}a_{j_1}a_{j_2}\cdots a_{j_s}\overline{D}(a_0\cdots\check a_{j_1}\cdots \check a_{j_s}\cdots a_m\big).\end{split}.\]
\end{proof}

\begin{pro}\label{der infi}
Let $M$ be an $A_{\infty}$-module. There exist a $A_{\infty}$-homomorphism \[\Der_k^m(A_{\infty},M)\to \Der_k^m(A,M\otimes_{A_{\infty}}B_{\infty}).\]
\end{pro}

\begin{proof}
Let $n,n'\in\N$. Since $A_n\hookrightarrow A_{n'}$ for all $n,n'\in\N$ such that $n\leq n'$, we have the following descending chain of homomorphism of $A$-modules:
\[\cdots\rightarrow \Der_k^m(A_n,M)\rightarrow \Der_k^m(A_{n-1},M)\rightarrow \cdots\rightarrow \Der_k^m(A_1,M)\rightarrow \Der_k^m(A_0,M).\] Hence we have an inverse system of $A$-modules $\{\Der_k^m(A_n,M)\}_{n\in\N}$.  We also have a natural ring homomorphism $\psi_n: B_n\hookrightarrow B_n^{\infty}$. Using $\psi_n$ we can define the $A$-homomorphism \[\begin{array}{ccccc}
\Der_k^m(A_{n},M)&\to& \Der_k^m(A,M\otimes_{A_n}B_n)&\to&\Der_k^m(A,M\otimes_{A_{\infty}}B_n^{\infty})\\
D&\mapsto &\overline{D}&\mapsto&    (Id_{M}\otimes\psi)\circ\overline{D}
\end{array}. \]
Moreover, using the natural inclusions $\theta_{n,n'}:B_n^{\infty}\hookrightarrow B_{n'}^{\infty}$ we can define a $A$-homomorphism $$\begin{array}{ccc}\Der_k^m(A,M\otimes_{A_{\infty}}B_n^{\infty})&\to&\Der_k^m(A,M\otimes_{A_{\infty}}B_{n'}^{\infty})\\
E&\mapsto&(Id_M\otimes\theta_{n,n'})\circ E\end{array}.$$ The previous homomorphisms induce the following commutative diagram of $A$-modules:
\[\begin{array}{ccccccc}
\cdots&\rightarrow&\Der_k^m(A_{n},M)&\rightarrow&\Der_k^m(A_{n-1},M)&\rightarrow&\cdots\\
 & & \downarrow \phi_n& & \downarrow \phi_{n-1}& & \\
\cdots&\to&\Der_k^m(A,M\otimes_{A_{\infty}}B_n^{\infty})&\to&\Der_k^m(A,M\otimes_{A_{\infty}}B_{n-1}^{\infty})&\to&\cdots
\end{array}.\] 
This diagram induces the $A$-homomorphism \[\begin{array}{ccc}
\overline{\phi_{\infty}}:\underset{n\in\N}{\varprojlim}\Der_k^m(A_{n},M)&\to&\underset{n\in\N}{\varprojlim}\Der_k^m(A,M\otimes_{A_{\infty}}B_n^{\infty})\\
(D_0,D_1,\ldots,D_n,\ldots)&\mapsto &(\overline{D_0},\overline{D_1},\ldots,\overline{D_n},\ldots)
\end{array}.\] Applying Theorem \ref{upd} we obtain
\[
\overline{\phi'_{\infty}}:\underset{n\in\N}{\varprojlim}\Hom(\Omega_{A_n/k}^{(m)},M)\to\underset{n\in\N}{\varprojlim}\Hom(\Omega_{A/k}^{(m)},M\otimes_{A_{\infty}}B_n^{\infty}).\]
By the commutative properties of inverse limits with $\Hom$ functor and tensor products (see \cite[Chapter 2]{KS}), we get
\[
\phi'_{\infty}:\Hom(\underset{n\in\N}{\varinjlim}\Omega_{A_n/k}^{(m)},M)\to\Hom(\Omega_{A/k}^{(m)},M\otimes_{A_{\infty}}\underset{n\in\N}{\varprojlim}B_n^{\infty}).\] 
Using Remark \ref{Blim} and equation (\ref{Omeinf}) results in 
\[
\phi'_{\infty}:\Hom(\Omega_{A_{\infty}/k}^{(m)},M)\to\Hom(\Omega_{A/k}^{(m)},M\otimes_{A_{\infty}}B_{\infty}).\]
Finally, by Theorem \ref{upd} we obtain a homomorphism of $A$-modules
\[\phi_{\infty}:\Der_k^m(A_{\infty},M)\to\Der_k^m(A,M\otimes_{A_{\infty}}B_{\infty}).\]
By Lemma \ref{An-A}, $\Der_k^m(A,M\otimes_{A_{\infty}}B_{\infty})$ is a $A_n$-module for all $n\in\N$. In particular, it is a $A_{\infty}$-module. On the other hand, by definition $\Der_k^m(A_{\infty},M)$ is a $A_{\infty}$-module. Thus, $\phi_{\infty}$ is a homomorphism of $A_{\infty}$-modules.
\end{proof}
\subsection{Proof of Theorem \ref{mT} $(2)$}

\begin{pro}\label{der sur}
Let $n\in\N$, $m\in\N_{\geq1}$. Let $\K$ be a field of characteristic zero and $A=\K[x_1,\ldots,x_s]/I$. 
In this case, the map $\phi_n$ of Proposition \ref{der m} is surjective. 
\end{pro}

Let us introduce some notation that will be used in the proof of this result. Let $\alpha=(\alpha^0_1,\ldots,\alpha^n_1,\ldots,\alpha^0_s,\ldots,\alpha^n_s)\in \N^{s(n+1)}$ and $\beta=(\beta_1\ldots,\beta_s)\in\N^{s}$.
\begin{itemize}
\item $\underline{x}^{\alpha}:=(x_1^{(0)})^{\alpha^0_1}(x_1^{(1)})^{\alpha^1_1} \cdots(x_1^{(n)})^{\alpha^n_1}\cdots (x_i^{(j)})^{\alpha_i^j}\cdots(x_s^{(0)})^{\alpha^0_s}\cdots(x_s^{(n)})^{\alpha^n_s}.$

\item $x^{\beta}:=x_1^{\beta_1}\cdots x_s^{\beta_s}.$

\item $|\alpha|:=\sum_{i=1}^s\sum_{j=0}^n\alpha_i^j$ and $|\beta|:=\sum_{i=1}^s\beta_i.$

\item Set $\Delta:=\{\underline{x}^{\alpha}| \alpha\in\N^{s(n+1)} \textnormal{ and }1\leq |\alpha|\leq m\}.$ For all $i\in\{1,\ldots,s\}$ and $j\in\{1,\ldots,n\}$, we define the following subsets of $\Delta$:
\begin{align}
\Delta^0&:=\left\{\underline{x}^{\alpha}=(x_1^{(0)})^{\alpha^0_1}(x_2^{(0)})^{\alpha^0_2}\cdots(x_s^{(0)})^{\alpha^0_s}|\underline{x}^{\alpha}\in\Delta\right\}.\notag\\
\Delta_i^{j}&:=\left\{\underline{x}^{\alpha}=x_i^{(j)}(x_i^{(0)})^{\alpha^0_i}(x_{i+1}^{(0)})^{\alpha^0_{i+1}}\cdots(x_s^{(0)})^{\alpha^0_s}|\underline{x}^{\alpha}\in\Delta\right\}.\notag
\end{align}

\item Let $\gamma=(\gamma_1,\ldots,\gamma_{|\beta|})\in\N^{|\beta|}.$ We set $$x^{(\gamma)}:=x_1^{(\gamma_1)}\cdots x_1^{(\gamma_{\beta_1})}x_2^{(\gamma_{\beta_1+1})}\cdots x_2^{(\gamma_{\beta_1+\beta_2})}\cdots x_s^{(\gamma_{\beta1+\cdots+\beta_{s-1}+1})}\cdots x_s^{(\gamma_{|\beta|})}$$ 
(since there are no conditions on $\gamma_l$, we see $x^{(\gamma)}$ as an element of $k[x_1,\ldots,x_s]_{\infty}$).

\item Let $j\in\{0,1\ldots,n\}.$ We denote \[\Gamma^j:=\{\gamma\in\N^{|\beta|}|\mbox{  }|\gamma|=j,\gamma_l\neq j \textnormal{ for all } l\in\{1,\ldots,|\beta|\}\}.\]
\end{itemize}

\begin{lem}\label{g-a}Let $j\in\{0,1\ldots,n\}$ and $\gamma\in\Gamma^j$. Then there exists $\alpha_{\gamma}\in\N^{s(n+1)}$ such that $x^{(\gamma)}=\underline{x}^{\alpha_{\gamma}}$ and $|\beta|=|\alpha_\gamma|.$
\end{lem}
\begin{proof}
For the given $\gamma\in\Gamma_j$, we define the vector $\alpha_{\gamma}=(\alpha_i^j)_{\substack{1\leq i\leq s\\0\leq j\leq n}}\in\N^{s(n+1)},$ where  \begin{align}\alpha_1^j&=|\{l\in\{1,\ldots,\beta_1\}|\gamma_l=j\}|,\notag\\
\alpha_i^j&=|\{l\in\{\beta_1+\cdots+\beta_{i-1}+1,\ldots,\beta_1+\cdots+\beta_i\}|\gamma_l=j\}|, \textnormal{ for } i\geq 2.\notag\end{align}

By construction, we have that $x^{(\gamma)}=\underline{x}^{\alpha_{\gamma}}$. In addition, since $x^{(\gamma)}$ is a product of $|\beta|$ terms then $|\alpha_{\gamma}|=|\beta|.$
\end{proof}

\begin{rem}\label{xbeta}
Let $\beta\in\N^s.$ By iterating the Leibniz rule of Hasse-Schmidt derivations we obtain:
\begin{align}
(x^{\beta})^{(j)}&=\sum_{\substack{\gamma\in\N^{|\beta|}\\|\gamma|=j}}x^{(\gamma)}=\beta_1x_1^{(j)}(x_1^{(0)})^{\beta_1-1}(x_2^{(0)})^{\beta_2}\cdots(x_s^{(0)})^{\beta_s}+\sum_{\substack{\gamma\in\N^{|\beta|}\\|\gamma|=j\\ \gamma_l\neq j, \mbox{  } l=1,\ldots,\beta_1}}x^{(\gamma)}\notag\\&=\beta_1x_1^{(j)}(x_1^{(0)})^{\beta_1-1}(x_2^{(0)})^{\beta_2}\cdots(x_s^{(0)})^{\beta_s}+\beta_2x_2^{(j)}(x_1^{(0)})^{\beta_1}(x_2^{(0)})^{\beta_2-1}\cdots(x_s^{(0)})^{\beta_s}\notag\\&+\sum_{\substack{\gamma\in\N^{|\beta|}\\|\gamma|=j\\ \gamma_l\neq j,\mbox{  } l=1,\ldots, \beta_1+\beta_2}}x^{(\gamma)}.\notag
\end{align}
Iterating this process we obtain:
\[(x^{\beta})^{(j)}=\sum_{i=1}^s\beta_ix_i^{(j)}(x_1^{(0)})^{\beta_1}\cdots(x_i^{(0)})^{\beta_i-1}\cdots(x_s^{(0)})^{\beta_s}+\sum_{\gamma\in\Gamma^j}x^{(\gamma)}.\]
\end{rem}

\begin{rem}\label{xmono}
Let $B=k[y_1,\ldots,y_t]$ and $M$ be a $B$-module. By iterating the product rule of derivatives of order $m$ (\ref{ruleN}), we deduce that any derivation $D \in \Der_{B/k}^m(B,M)$ is determined by the monomials $x^{\tau}$ with $\tau \in \N^t$ such that $1 \leq |\tau| \leq m.$
\end{rem}

In the following example we illustrate the notation we introduced before as well as the main ideas of the proof of Proposition \ref{der sur}.  

\begin{exam}
Let $\K$ be a field of characteristic $0$. Set $A=\K[x_1,x_2]$, thus $A_1=\K[x_1^{(0)},x_1^{(1)},x_2^{(0)},x_2^{(1)}].$ Consider the map $\phi_1$ of Proposition \ref{der m}, \[\phi_1:\Der_{\K}^2(A_1,A_1)\to\Der_{\K}^2(A,A_1\otimes_{A_1}B_1).\]
Let $\Delta=\{\underline{x}^{\alpha}\mbox{}|\mbox{} \alpha\in\N^4 \mbox{} , \mbox{} 1\leq|\alpha|\leq 2\}.$ Notice that \begin{align} \Delta^{0}&=\{x_1^{(0)},x_2^{(0)},(x_1^{(0)})^2,x_1^{(0)}x_2^{(0)},(x_2^{(0)})^2\},\notag\\\Delta_{1}^1&=\{x_1^{(1)},x_1^{(0)}x_1^{(1)},x_1^{(1)}x_2^{(0)}\},\notag\\\Delta_{2}^1&=\{x_2^{(1)},x_2^{(1)}x_2^{(0)}\}.\notag \end{align}
Let $E\in\Der_{\K}^2(A,A_1\otimes_{A_1}B_1).$ Then $E=E_0\otimes1+E_1\otimes t,$ where $E_j\in\Der^m_{\K}(A,A_1)$ for $j=0,1.$ We proceed to define $D\in\Der_{\K}^2(A_1,A_1)$ such that $\phi_1(D)=E$. By Remark \ref{xmono} it is enough to define $D$ in the monomials $\underline{x}^\alpha$, with $\alpha\in\Delta$:
\[D(\underline{x}^{\alpha})=\left\{\begin{array}{lll}E_0(x_1^{\alpha_1^0}x_2^{\alpha_2^0})&\textnormal{ if } &\underline{x}^{\alpha}\in\Delta^0\\ \frac{1}{\alpha_{1}^0+1}E_{1}\left(x_{1}^{\alpha_{1}^0+1}x_{2}^{\alpha_{2}^0}\right)&\textnormal{ if }&\underline{x}^{\alpha}\in\Delta_1^{1}\\ \frac{1}{\alpha_{2}^0+1}E_{1}\left(x_{2}^{\alpha_{2}^0+1}\right)&\textnormal{ if }&\underline{x}^{\alpha}\in\Delta_2^{1}\\0 &\textnormal{ if }& \underline{x}^{\alpha}\in\Delta\setminus\Delta^{0}\cup\Delta_1^1\cup\Delta_2^1
\end{array}\right..\]

Now, we prove that $\overline{D}=E$. By Remark \ref{xmono}, it is enough to check this equality on the monomials $x^{\beta}$ for $\beta\in\N^2,$ such that $1\leq|\beta|\leq2.$ Note that
\begin{align}
\overline{D}(x_1)&=D\circ d_0(x_1)\otimes1+D\circ d_1(x_1)\otimes t\notag\\
&=D(x_1^{(0)})\otimes1+D(x_1^{(1)})\otimes t\notag\\
&=E_0(x_1)\otimes1+E_1(x_1)\otimes t=E(x_1).\notag\\\ \notag\\
\overline{D}((x_1)^2)&=D\circ d_0((x_1)^2)\otimes1+D\circ d_1((x_1)^2)\otimes t\notag\\
&=D((x_1^{(0)})^2)\otimes1+2D(x_1^{(0)}x_1^{(1)})\otimes t\notag\\
&=E_0((x_1)^2)\otimes1+1/2(2E_1((x_1)^2))\otimes t=E((x_1)^2).\notag\\ \notag\\
\overline{D}(x_1x_2)&=D\circ d_0(x_1x_2)\otimes1+D\circ d_1(x_1x_2)\otimes t\notag\\
&=D(x_1^{(0)}x_2^{(0)})\otimes1+D(x_1^{(0)}x_2^{(1)}+x_1^{(1)}x_2^{(0)})\otimes t\notag\\
&=E_0(x_1x_2)\otimes1+E_1(x_1x_2)\otimes t=E(x_1x_2).\notag
\end{align}
Similar computations show that $\overline{D}(x_2)=E(x_2)$ and $\overline{D}((x_2)^2)=E((x_2)^2)$. Therefore $\phi_1(D)=\overline{D}=E.$
\end{exam}

\begin{proof}{(of Proposition \ref{der sur})} Let $E=\sum_{j=0}^nE_j\otimes t^j\in\Der_k^m(A,M\otimes B_n).$ By Remark \ref{xmono}, it is enough to find a derivation $D\in\Der_k^m(A_n,M)$ such that $\overline{D}(x^{\beta})=E(x^{\beta})$, for all monomials $x^{\beta}$ such that $1\leq |\beta|\leq m$. 
We define a derivation $D:A_n\to M$ by  determining its values in $\underline{x}^{\alpha}\in\Delta$ as follows:

\[D(\underline{x}^{\alpha})=\left\{\begin{array}{lll}E_0(x_1^{\alpha_1^0}\cdots x_s^{\alpha_s^0})&\textnormal{ if } &\underline{x}^{\alpha}\in\Delta^0\\ \frac{1}{\alpha_{i}^0+1}E_{j}\left(x_{i}^{\alpha_{i}^0+1}x_{i+1}^{\alpha_{i+1}^0}\cdots x_{s}^{\alpha_{s}^0}\right)&\textnormal{ if }&\underline{x}^{\alpha}\in\Delta_i^{j}\textnormal{ for some } j\in\{1,\ldots,n\}\\0 &\textnormal{ if }& \underline{x}^{\alpha}\notin\Delta_i^{j} \textnormal{ for all } j\in\{0,\ldots,n\}
\end{array}\right..\]
Let $\beta\in\N^s$ such that $1\leq |\beta|\leq m$ and $D$ as before. By definition of $\overline{D}$ and applying Remark \ref{xbeta} we obtain 

\begin{align}
\overline{D}(x^{\beta})&=\sum_{j=0}^n\big(D\circ d_j\big)(x^{\beta})\otimes t^j=\sum_{j=0}^nD\big((x^{\beta})^{(j)}\big)\otimes t^j\notag\\&=D((x^{\beta})^{(0)})\otimes 1+\sum_{j=1}^n\left[D((x^{\beta})^{(j)})\right]\otimes t^j\notag\\&=D\left((x_1^{(0)})^{\beta_1}(x_2^{(0)})^{\beta_2}\cdots (x_s^{(0)})^{\beta_s}\right)\otimes 1\notag\\&+\sum_{j=1}^n\left[D\left(\sum_{i=1}^s\beta_ix_i^{(j)}(x_1^{(0)})^{\beta_1}\cdots (x_i^{(0)})^{\beta_i-1}\cdots (x_s^{(0)})^{\beta_s}   \right)+\sum_{\gamma\in\Gamma^j}D(x^{(\gamma)})\right]\otimes t^j\notag
\end{align}

Firstly, for $j=0$ and $(x_1^{(0)})^{\beta_1}(x_2^{(0)})^{\beta_2}\cdots (x_s^{(0)})^{\beta_s}\in\Delta^0$, we have that\[D\left((x_1^{(0)})^{\beta_1}(x_2^{(0)})^{\beta_2}\cdots (x_s^{(0)})^{\beta_s}\right)=E_0(x_1^{\beta_1}\cdots x_s^{\beta_s})=E_0(x^{\beta}).\]

 Consider $j\geq1$. We observe:
\begin{itemize}

\item For $x_{i}^{(j)}(x_{i}^{(0)})^{\beta_{i}}\cdots (x_s^{(0)})^{\beta_s}\in\Delta_i^j$, \[D\left(\beta_{i}x_{i}^{(j)}(x_{i}^{(0)})^{\beta_{i}-1}\cdots (x_s^{(0)})^{\beta_s}\right) =\beta_{i}\frac{1}{\beta_{i}}E_j(x_{i}^{\beta_{i}-1+1}\cdots x_s^{\beta_s})=E_j(x^{\beta}).\]

\item For each $\beta\in\N^s$, we define $i_1=\min\{i|\beta_i\neq0, i\in\{1,\ldots,s\}\}.$ For $i>i_1$ we have that $D\left(\beta_{i}x_{i}^{(j)}(x_{i_1}^{(0)})^{\beta_{i_1}}\cdots(x_{i}^{(0)})^{\beta_{i}-1}\cdots (x_s^{(0)})^{\beta_s}\right)=0$.

\item By Lemma \ref{g-a}, for $\gamma\in\Gamma^j$ there exists $\alpha_{\gamma}\in\N^{s(n+1)}$ such that $x^{(\gamma)}=\underline{x}^{\alpha_{\gamma}}$ and $|\alpha_\gamma|=|\beta|.$ Since $1\leq|\beta|\leq m$ then $\underline{x}^{\alpha_{\gamma}}\in\Delta.$ In addition, since $\gamma_l\neq j$ for all $l\in\{1,\ldots,|\beta|\}$ we have that $\underline{x}^{\alpha_{\gamma}}\not\in\Delta_i^j$ for all $j\in\{0,\ldots,n\}.$ Then for $\gamma\in\Gamma^j$ we have $D(x^{(\gamma)})=D(\underline{x}^{\alpha_{\gamma}})=0.$
\end{itemize}

By the three previous items, we conclude that
\begin{align}
D\left((x^{\beta})^{(j)}\right)&=D\left(\beta_{i_1}x_{i_1}^{(j)}(x_{i_1}^{(0)})^{\beta_{i_1}-1}\cdots (x_s^{(0)})^{\beta_s}\right)\notag\\ &+\sum_{i>i_1}^sD\left(\beta_{i}x_{i}^{(j)}(x_{i}^{(0)})^{\beta_{i}-1}\cdots (x_s^{(0)})^{\beta_s}\right)+\sum_{\gamma\in\Gamma^j}D(x^{\gamma})\notag\\&=E_j(x^{\beta}).\notag
\end{align}

Consequently
\begin{align}
\overline{D}\left(x^{\beta}\right)=\sum_{j=0}^nD\big((x^{\beta})^{(j)}\big)\otimes t^j=\sum_{j=0}^nE_j\big(x^{\beta}\big)\otimes t^j=E\left(x^{\beta}\right)\notag.
\end{align}

\end{proof}
\subsection{Proof of Theorem \ref{mT} $(3)$}
Throughout this section, $\K$ denotes a field of characteristic zero. The following example shows that for some ring $A$ the map $\phi_n$ is not injective, for any $m\geq2$ and $n\geq1$.
\begin{exam}\label{notiny}
Let $A=\K[x]$. Thus, $A_n=\K[x_0,x_1,x_2,\ldots,x_n]$. Let $M=A_n$ and consider $D=\frac{1}{m!}\frac{\partial^m }{\partial x_n^m}\in\Der^m_{k}(A_n,A_n)$. Let us show that $\overline{D}(f)=0$ for all $f\in A$. Indeed, since $D$ is additive it is enough to show that $D((x^l)^{(i)})=0$ for all $l\in\N$ and $i\in\{0,\ldots,n\}$. This is true if $i<n$ since no monomial in $(x^l)^{(i)}$ contains the variable $x_n$. For $i=n$, a straightforward induction on $l$ using the Leibniz rule on $(x^l)^{(n)}$ gives $D((x^l)^{(n)})=0$, for $m\geq2$.
\end{exam}

The previous example only holds for the partial derivative with respect to $x_n$, that is, the ``last" variable of $A_n$. This may lead to believe that in $A_{\infty}$ the map $\phi_{\infty}$ is injective. Unfortunately, it is not the case, as we show in what follows.

\begin{lem}\label{B2}
Let $A=\K[x_1,\ldots,x_s]$ and let $x^{\beta}$ be a monomial in $A.$ We have that \[d_A^2(x^{\beta})=\sum_{\substack{\alpha\in\N^{s}\\|\alpha|=2}}\frac{1}{\alpha!}\frac{\partial^{\alpha}}{\partial x^{\alpha}}(x^{\beta})d_A^2(x^{\alpha})+(2-|\beta|)\sum_{\substack{\alpha\in\N^{s}\\|\alpha|=1}}\frac{1}{\alpha!}\frac{\partial^{\alpha}}{\partial x^{\alpha}}(x^{\beta})d_A^2(x^{\alpha}).\]
\end{lem}
\begin{proof}
Firstly,  let $e_1,\ldots,e_s$ be the canonical basis of $\N^s$. Denote $\Delta_{\alpha}(x^{\beta})=\frac{1}{\alpha!}\frac{\partial^{\alpha}}{\partial x^{\alpha}}(x^{\beta})$ and $(d_A^2(x))^{\alpha}=d_A^2(x_1)^{\alpha_1}\cdots d_A^2(x_s)^{\alpha_s}$.  By \cite[Chapter II-2]{N}, we have that $$d_A^2(x^\beta) =\sum_{\substack{\alpha\in\N^{s}\\1\leq|\alpha|\leq2}}\Delta_{\alpha}(x^{\beta})(d_A^2(x))^{\alpha}.$$ In particular, if $|\beta|=2$, we deduce that 
$$(d_A^2(x))^{\beta}=d_A^2(x^{\beta})-\sum_{i=1}^s\Delta_{e_i}(x^{\beta})d^2_A(x_i).$$

Now let $x^{\beta}\in A$. We observe that:
\begin{align}
d_A^2(x^{\beta})&=\sum_{\substack{\alpha\in\N^{s}\\1\leq|\alpha|\leq2}}\Delta_{\alpha}(x^{\beta})(d_A^2(x))^{\alpha}\notag\\&=\sum_{\substack{\alpha\in\N^{s}\\|\alpha|=2}}\Delta_{\alpha}(x^{\beta})\Big(d_A^2(x^{\alpha})-\sum_{i=1}^s\Delta_{e_i}(x^{\alpha})d_A^2(x_i)\Big)\notag\\&+\sum_{\substack{\alpha\in\N^{s}\\|\alpha|=1}}\Delta_{\alpha}(x^{\beta})(d_A^2(x))^{\alpha}\notag\\&=\sum_{\substack{\alpha\in\N^{s}\\|\alpha|=2}}\Delta_{\alpha}(x^{\beta})d_A^2(x^{\alpha})-\sum_{\substack{\alpha\in\N^{s}\\|\alpha|=2}}\sum_{i=1}^s\Delta_{\alpha}(x^{\beta})\Delta_{e_i}(x^{\alpha})d_A^2(x_i)\notag\\&+\sum_{i=1}^s\Delta_{e_i}(x^{\beta})d_A^2(x_i)\notag\\&=\sum_{\substack{\alpha\in\N^{s}\\|\alpha|=2}}\Delta_{\alpha}(x^{\beta})d_A^2(x^{\alpha})-\sum_{\substack{\alpha\in\N^{s}\\|\alpha|=2}}\sum_{i=1}^s\binom{\beta-e_i}{\alpha-e_i}\Delta_{e_i}(x^{\beta})d_A^2(x_i)\notag\\&+\sum_{i=1}^s\Delta_{e_i}(x^{\beta})d_A^2(x_i)\notag\\&=\sum_{\substack{\alpha\in\N^{s}\\|\alpha|=2}}\Delta_{\alpha}(x^{\beta})d_A^2(x^{\alpha})+\sum_{i=1}^s\Big(1-\sum_{\substack{\alpha\in\N^{s}\\|\alpha|=2}}\binom{\beta-e_i}{\alpha-e_i}\Big)\Delta_{e_i}(x^{\beta})d_A^2(x_i)\notag\\&=\sum_{\substack{\alpha\in\N^{s}\\|\alpha|=2}}\Delta_{\alpha}(x^{\beta})d_A^2(x^{\alpha})+\sum_{i=1}^s\Big(1-\beta_1-\cdots-(\beta_i-1)-\cdots-\beta_s\Big)\Delta_{e_i}(x^{\beta})d_A^2(x_i)\notag\\&=\sum_{\substack{\alpha\in\N^{s}\\|\alpha|=2}}\Delta_{\alpha}(x^{\beta})d_A^2(x^{\alpha})+\sum_{i=1}^s\Big(2-|\beta|\Big)\Delta_{e_i}(x^{\beta})d_A^2(x_i)\notag\\&=\sum_{\substack{\alpha\in\N^{s}\\|\alpha|=2}}\frac{1}{\alpha!}\frac{\partial^{\alpha}}{\partial x^{\alpha}}(x^{\beta})d_A^2(x^{\alpha})+(2-|\beta|)\sum_{\substack{\alpha\in\N^{s}\\|\alpha|=1}}\frac{1}{\alpha!}\frac{\partial^{\alpha}}{\partial x^{\alpha}}(x^{\beta})d_A^2(x^{\alpha}).\notag
\end{align}
\end{proof}

\begin{rem}\label{pD}
Let $A=\K[x]$ and $n\in \N$. Using Lemma \ref{B2} and Theorem \ref{upd}, we obtain that if $D\in\Der_k^2(A_n,A_n)$ and $x^{\beta}$ is a monomial in $A_n$, then \[D(x^{\beta})=\sum_{\substack{\alpha\in\N^{n+1}\\|\alpha|=2}}F_{\alpha}\frac{1}{\alpha!}\frac{\partial^{\alpha}}{\partial x^{\alpha}}(x^{\beta})+(2-|\beta|)\sum_{\substack{\alpha\in\N^{n+1}\\|\alpha|=1}}F_{\alpha}\frac{1}{\alpha!}\frac{\partial^{\alpha}}{\partial x^{\alpha}}(x^{\beta}),\] where $F_{\alpha}=D(x^{\alpha})\in A_n.$ 
\end{rem}

\begin{lem}\label{der ker}
Consider the notation of Remark \ref{pD}. 
Then, $$\ker\phi_n=\left\{D\in\Der_k^2(A_n,A_n)|\mbox{ }F_{e_j}=0,\mbox{  }\sum_{i=0}^jF_{e_i+e_{j-i}}=0, \textnormal{ for all }j\in\{0,1,\ldots,n\}\right\},$$ where $e_0,e_1,\ldots,e_n$ is the canonical basis of $\N^{n+1}.$
\end{lem}
\begin{proof} 
Denote as $\mathcal{A}$ the set on the right in the statement of the lemma.

First, we prove that if $D\in\mathcal{A}$ then $D\in\ker\phi_n.$ By Remark \ref{xmono}, we know that elements of $\Der^2_k(A, A_n\otimes B_n)$ are determined by the monomials $x,x^2.$ Thus, it is enough to prove that $\phi_n(D)(x)=0$ and $\phi_n(D)(x^2)=0$. Note that:\begin{align}
\phi_n(D)(x)&=\sum_{j=0}^n(D(x^{(j)}))\otimes t^j\notag\\
&=\sum_{j=0}^n\sum_{\substack{\alpha\in\N^{n+1}\\|\alpha|=2}}F_{\alpha}\frac{1}{\alpha!}\frac{\partial^{\alpha}}{\partial x^{\alpha}}(x^{(j)})+(2-1)\sum_{\substack{\alpha\in\N^{n+1}\\|\alpha|=1}}F_{\alpha}\frac{1}{\alpha!}\frac{\partial^{\alpha}}{\partial x^{\alpha}}(x^{(j)})\otimes t^j\notag\\
&=\sum_{j=0}^nF_{e_j}\otimes t^j=0.\notag
\end{align}
\begin{align}
\phi_n(D)(x^2)&=\sum_{j=0}^n\Big(D((x^2)^{(j)})\Big)\otimes t^j\notag\\
&=\sum_{j=0}^n\Big(D(\sum_{i=0}^jx^{(i)}x^{(j-i)})\Big)\otimes t^j\notag\\
&=\sum_{j=0}^n\sum_{\substack{\alpha\in\N^{n+1}\\|\alpha|=2}}F_{\alpha}\frac{1}{\alpha!}\frac{\partial^{\alpha}}{\partial x^{\alpha}}(\sum_{i=0}^jx^{(i)}x^{(j-i)})\notag\\&+(2-2)\sum_{\substack{\alpha\in\N^{n+1}\\|\alpha|=1}}F_{\alpha}\frac{1}{\alpha!}\frac{\partial^{\alpha}}{\partial x^{\alpha}}(x^{(j)})(\sum_{i=0}^jx^{(i)}x^{(j-i)})\otimes t^j\notag\\
&=\sum_{j=0}^n\sum_{i=0}^j\sum_{\substack{\alpha\in\N^{n+1}\\|\alpha|=2}}F_{\alpha}\frac{1}{\alpha!}\frac{\partial^{\alpha}}{\partial x^{\alpha}}(x^{(i)}x^{(j-i)})\otimes t^j\notag\\
&=\sum_{j=0}^n\sum_{i=0}^jF_{e_i+e_{j-i}}\otimes t^j=0.\notag
\end{align}
Therefore, $D\in\ker\phi_n.$
Now, we prove the other inclusion. Set $\partial_{x^{\alpha}}:=\frac{\partial^{\alpha}}{\partial x^{\alpha}}$. Let $D\in\ker\phi_n$, note that 
\begin{align}
0&=\overline{D}=\sum_{j=0}^n D\circ d_j\otimes t^j\notag
\end{align}
Then for all $f\in A$ we obtain \begin{align}\label{eqker}0=D\big(d_j(f)\big)\otimes t^j .
\end{align} 
Consider the $A_n$-bilinear map $L:A_n\times B_n \to B_n, (F,G)\mapsto FG.$ The equation (\ref{eqker}) implies that \begin{align}\label{eq2ker}
\sum_{j=0}^nL\Big(D\big(d_j(f)\big),t^j\Big)=0, \mbox{  } \textnormal{for all } f\in A.
\end{align}
By equation (\ref{eq2ker}) we obtain that
\begin{align}
\sum_{j=0}^nD\big(d_j(f)\big)t^j=P \mbox{  } \textnormal{for some }P\in\langle t^{n+1}\rangle.
\end{align}
Hence $D(d_j(f))=0$ for all $j\in\{0,\ldots,n\}.$ In particular if $f=x$ and using Remark \ref{pD} then 
\begin{equation*}
0=\big(\sum_{\substack{\alpha\in\N^{n+1}\\|\alpha|=2}}F_{\alpha}\frac{1}{\alpha!}\frac{\partial^{\alpha}}{\partial x^{\alpha}}(x^{(j)})+(2-1)\sum_{\substack{\alpha\in\N^{n+1}\\|\alpha|=1}}F_{\alpha}\frac{1}{\alpha!}\frac{\partial^{\alpha}}{\partial x^{\alpha}}(x^{(j)})\big)=F_{e_j}
\end{equation*}for all $j\in\{0,1,\ldots,n\}$.  Hence, $$D=\sum_{\substack{\alpha\in\N^{n+1}\\|\alpha|=2}}F_{\alpha}\frac{1}{\alpha!}\partial x^{\alpha}.$$
Now, if $f=x^2$ thus
\begin{align}\notag
0&=\big(\sum_{\substack{\alpha\in\N^{n+1}\\|\alpha|
=2}}F_{\alpha}\frac{1}{\alpha!}\partial_{x^{\alpha}}\big)\big(d_j(x^2)\big)\\\notag
&=\big(\sum_{\substack{\alpha\in\N^{n+1}\\|\alpha|
=2}}F_{\alpha}\frac{1}{\alpha!}\partial_{x^{\alpha}}\big)\big(\sum_{i=0}^jd_i(x)d_{j-i}(x)\big).
\end{align} 
Note that, for $\alpha\in\N^{n+1}$ such that $|\alpha|=2$
\[\frac{1}{\alpha!}\partial_{x^{\alpha}}\big(x^{(i)}x^{(j-i)}\big)=\left\{\begin{array}{ll}0,& \textnormal{if } \alpha\neq e_i+e_{j-i}\\
1,& \textnormal{if } \alpha=e_i+e_{j-i}
\end{array}\right..\]
Hence,
\begin{align}
0&=\big(\sum_{\substack{\alpha\in\N^{n+1}\\|\alpha|
=2}}F_{\alpha}\frac{1}{\alpha!}\partial_{x^{\alpha}}\big)\big(\sum_{i=0}^nd_i(x)d_{j-i}(x)\big)=\sum_{i=0}^jF_{e_i+e_{j-i}}.\notag
\end{align} 
\end{proof}

The following example shows that the map $\phi_{\infty}$ is not injective.
\begin{exam}
Let $A=\K[x]$. Thus, $A_{\infty}=\K[x_0\ldots,x_n,\ldots].$ For each $k\in\N$ define derivations $D_k:A_k\to A_{\infty}$ as follows:
\begin{align}
D_0&=0\notag\\
D_1&=\frac{1}{2!}\frac{\partial^2}{\partial x_1^2}\notag\\
D_2&=D_1-\frac{1}{2}\frac{\partial^2}{\partial x_0x_2}\notag\\
D_k&=D_{k-1}+\frac{\partial^2}{\partial x_0x_k}-\frac{\partial^2}{\partial x_1x_{k-1}}\textnormal{ for }k\geq 3 \notag
\end{align}
We claim that $D=(D_0,D_1,D_2,\ldots)\in\ker\phi_{\infty}$. This is a consequence of the following facts:
\begin{itemize}
\item[(1)] $D=(D_0,D_1,D_2,\ldots)\in\underset{n\in\N}{\varprojlim}\Der_k^m(A_{n},A_{\infty}).$ 
\item[(2)] $\phi_k(D_k)=0$ for all $k\in\N$. 
\end{itemize}

To prove (1), fix $k\in\{0,1,\ldots,n\}$ and let $j< k$. Since monomials in $A_j$ do not have the variable $x_k$ we have $\frac{\partial^2}{\partial x_0x_k}|_{A_j}=0.$ On the other hand, by the Leibniz rule of Hasse-Schmidt derivations we obtain $\frac{\partial^2}{\partial x_1x_{k-1}}|_{A_j}=0.$ Hence, $D_k|_{A_j}=D_j.$ 

Now we prove (2). %
Note that $\phi_0(D_0)=0$ and, by example \ref{notiny}, $\phi_1(D_1)=0$. Using the Lemma \ref{der ker} we prove that $\phi_i(D_i)=0$ for all $i\geq2.$ Observe that $D_k$ is defined by partial derivatives of order 2. Hence,
\[F^k_{e_j}=D_k(x_j)=0.\]
Now we prove $\sum_{i=0}^j F^k_{e_i+e_{j-i}}=0$ for all $k$ and all $j\in\{0,\ldots,k\}$. Firstly, we consider $D_2$. Observe that 
\begin{align}
F^2_{e_0+e_0}&=F^2_{2e_0}=D_2(x_0^2)=0,\notag\\
F^2_{e_0+e_1}+F^2_{e_1+e_0}&=2F^2_{e_0+e_1}=2D_2(x_0x_1)=0,\notag\\
F^2_{e_0+e_2}+F^2_{e_1+e_1}+F^2_{e_2+e_0}&=2F^2_{e_0+e_2}+F^2_{2e_1}=2D_2(x_0x_2)+D_2(x_1^2)=1-1=0.\notag
\end{align}
By Lemma \ref{der ker}, $\phi_2(D_2)=0.$ We continue by induction on $k\geq3.$ Set $\frac{\partial^2}{\partial x_ix_j}=\partial_{x_ix_j}.$ By definition, $D_3=D_2+\partial_{x_0x_3}-\partial_{x_1x_2}.$ Note that:
\begin{align}
F^3_{e_0+e_0}&=F^3_{2e_0}=D_3(x_0^2)=0,\notag\\
F^3_{e_0+e_1}+F^3_{e_1+e_0}&=2F^3_{e_0+e_1}=2D_3(x_0x_1)=0,\notag\\
F^3_{e_0+e_2}+F^3_{e_1+e_1}+F^3_{e_2+e_0}&=2D_3(x_0x_2)+D_3(x_1^2)=1-1=0.\notag\\
F^3_{e_0+e_3}+F^3_{e_1+e_2}+F^3_{e_2+e_1}+F^3_{e_3+e_0}&=2D_3(x_0x_3)+2D_3(x_1x_2)=2-2=0.\notag
\end{align}
Thus, $\phi_3(D_3)=0$. Assume that $D_k\in\ker\phi_k.$ We already know that $D_{k+1}|_{A_j}=D_j,$ for $j\in\{0,\ldots,k\}.$ Using the hypothesis of induction for $j\in\{3,4,\ldots,n\}$ we have that \[\sum_{i=0}^jF^{k+1}_{e_i+e_{j-i}}=0.\] For $j=k+1$ observe that 
\begin{align}
\sum_{i=0}^{k+1}F^{k+1}_{e_i+e_{j-i}}=\sum_{i=0}^{k+1}D_{k+1}(x_ix_{j-i})=2D_{k+1}(x_0x_{k+1})+2D_{k+1}(x_1x_k)=2-2=0.\notag
\end{align}
By Lemma \ref{der ker}, $\phi_{k+1}(D_{k+1})=0$. Therefore, $D_k\in \ker\phi_k$ for all $k\in\N.$ 

\end{exam}

\section{A map between  $\Omega_{A/k}^{(m)}\otimes_{A_n}P_n$ and $\Omega_{A_n/k}^{(m)}$}

In this section we study an implication of Theorem \ref{mT} concerning the module of high-order differentials. 

\begin{defi}\cite[Section 4]{dFDo1}
For any $n\in\N\cup\{\infty\},$ we define $P_n$ to be the $B_n$-module given by
\[P_n:=t^{-n}A_n[t]/tA_n[t]\] when $n$ is finite and \[P_{\infty}:=A_{\infty}((t))/tA_{\infty}[[t]]\] when $n=\infty$.
\end{defi}

The main consecuence that was given in \cite{dFDo1} of Theorem \ref{Der1} is the following formula:

\begin{teo}\cite[Theorem 5.3]{dFDo1} Let $A$ be a $k$-algebra. For all $n\in\N\cup\{\infty\}$ there exist an isomorphism of $A_n$-modules
\[\Omega_{A_n/k}\simeq\Omega_{A/k}\otimes_{A_n}P_n.\]
\end{teo}

Our next goal is to study a relation between the modules $\Omega_{A_n/k}^{(m)}$ and $\Omega_{A/k}^{(m)}$ using Theorem \ref{mT}.

\begin{lem}\cite[Lemma 4.5]{dFDo1}
For $n\in\N$, the morphism that sends $t^{-j}$ to $t^{-j+n}$ gives an isomorphism of $B_n$-modules between $P_n$ and $B_n$. By contrast, $P_{\infty}$ and $B_{\infty}$ are not isomorphism, not even as $A_{\infty}$-modules.
\end{lem}

\begin{rem}\cite[Remark 4.3]{dFDo1}\label{HomPn}
 For every $A_n$-module $M$,  there is a canonical isomorphism\[M\otimes_{A_n}B_n\simeq\Hom_{A_n}(P_n,M)\] as a $B_n$-modules given by \[G=\sum_{j=0}^nm_j\otimes a_jt^{j}\mapsto\big(\theta_G:p=\sum_{j=0}^na'_{-j}t^{-j}\mapsto\sum_{j=0}^na'_{-j}a_jm_j\big).\]
\end{rem}

Using the map of Proposition \ref{der m}, we can describe an explicit map between $\Omega^{(m)}_{A/k}\otimes_A P_n$ and $\Omega^{(m)}_{A_n/k}$.

\begin{pro}\label{tensor} Let $n\in\N\cup\{\infty\}$. Let $A$ be a $k$-algebra. There exists a homomorphism of $A_n$-modules 
\begin{align*}
\phi_n^{\vee}:\Omega^{(m)}_{A/k}\otimes_A P_n \to\Omega_{A_n/k}^{(m)},
\end{align*}
such that $\phi_n^{\vee}(d_A^m(f)\otimes t^{-j})=d_{A_n}^m(f^{(j)}).$

\end{pro}
\begin{proof}Putting together Theorem \ref{upd}, Proposition \ref{der m}, Remark \ref{HomPn}, and Lemma \ref{An-A}, we obtain a chain of $A_n$-homomorphisms
\begin{align*}
\Hom_{A_n}(\Omega_{A_n/k}^{(m)},M)\cong\Der_{k}^m(A_n,M)\rightarrow&\Der_{k}^m(A,M\otimes_{A_n} B_n)\\ &\cong\Hom_A(\Omega_{A/k}^{(m)},M\otimes_{A_n} B_n)\\&\cong\Hom_A(\Omega_{A/k}^{(m)},\Hom_{A_n}(P_n,M))\\&\cong\Hom_{A_n}(\Omega_{A/k}^{(m)}\otimes_A P_n,M).
\end{align*}

All these homomorphisms are explicit. 
Applying them starting at the identity of $M=\Omega_{A_n/k}^{(m)}$, we obtain a homomorphism $\phi_n^{\vee}:\Omega_{A/k}^{(m)}\otimes P_n\to \Omega_{A_n/k}^{(m)},$ such that $\phi_n^{\vee}(d_A^m(f)\otimes t^{-j})=d_{A_n}^m(f^{(j)}).$
\end{proof}

\begin{rem}
 For usual derivations, the map $\phi_n^{\vee}$ appears implicitly in the proof of \cite[Lemma 5.1]{dFDo1}. In this case, $\phi^{\vee}_n$ is an isomorphism (see \cite[Theorem 5.3]{dFDo1}). On the contrary, by Theorem \ref{mT} (3), the map $\phi_n$ is not injective in general. In particular, by Theorem \ref{upd}, $\phi^{\vee}_n$ is not surjective in general. In the specific case where A is a polynomial ring, it is worth pointing out that the modules $\Omega_{A/k}^{(m)}\otimes P_n$ and $\Omega_{A_n/k}^{(m)}$ are free modules, but they have different ranks (see \cite[Remark 2.20]{BD}).\end{rem}

The following example shows that the homomorphism $\phi^{\vee}_n$ is not injective in general.

\begin{exam}
Let $A=\frac{\K[x_1,x_2]}{\langle x_1x_2\rangle}.$ Thus, $A_1=\frac{\K[x_1^{(0)},x_1^{(1)},x_2^{(0)},x_2^{(1)}]}{\langle x_1^{(0)}x_2^{(0)}, x_1^{(0)}x_2^{(1)}+x_1^{(1)}x_2^{(0)}\rangle}.$ Let us show that in this case $\phi^{\vee}_1$ is not injective. Set $f=x_1x_2$,  $f_1=x_1^{(0)}x_2^{(0)},$ $f_2=x_1^{(0)}x_2^{(1)}+x_1^{(1)}x_2^{(0)}.$ 

We consider the following presentation of $\Omega_{A/\K}^{(2)}\otimes_{A_1}P_1$ (see \cite[Corollary 2.15]{BD} and use the fact that $P_1\simeq B_1$ as $B_1$-modules):
$$\Omega^{(2)}_{A/\K}\otimes_{A_1}P_1=\frac{\bigoplus_{\substack{\alpha\in\N^{2}\\1\leq|\beta|\leq2}}A_1(d_A^2(x))^{\alpha}\otimes 1\bigoplus_{\substack{\alpha\in\N^{2}\\1\leq|\beta|\leq2}}A_1(d_A^2(x))^{\alpha}\otimes t^{-1}}{\langle f_{\beta}\otimes1, f_{\beta}\otimes t \rangle_{\substack{\beta\in\N^{2}\\0\leq|\beta|\leq1}}},$$ 
where $f_{\beta}=(d_A^{2}(x))^{\beta}d_A^{2}(f)$.

Similarly, consider the following presentation of $\Omega_{A_1/\K}^{(2)}$ (see \cite[Theorem 2.8]{BD1}):
$$\Omega^{(2)}_{A_1/\K}=\frac{\bigoplus_{\substack{\alpha\in\N^{4}\\1\leq|\beta|\leq2}}d_{A_1}^2(\underline{x})^{\alpha}}{\langle F^1_{\beta},F^2_{\beta} \rangle_{\substack{\beta\in\N^{4}\\0\leq|\beta|\leq1}}},$$
where $F^i_{\beta}=d_{A_1}^2(\underline{x})d_{A_1}^2(f_i), i\in\{1,2\}$.

Set $F=2\Big((d_A^2(x_2))^2\otimes x_1^{(0)}x_1^{(1)}\Big)+\frac{1}{2}\Big((d_A^2(x_2))^2\otimes (x_1^{(0)})^2 t^{-1}\Big).$ First we see that $F\neq 0.$ Suppose that $F=0$. This implies 
\begin{align}\label{F=0}F-(\sum_{\substack{\beta\in\N^{2}\\0\leq|\beta|\leq1}}g^0_\beta f_{\beta}\otimes 1+\sum_{\substack{\beta\in\N^{2}\\0\leq|\beta|\leq1}}g^1_\beta f_{\beta}\otimes t^{-1})=0.
\end{align} 
In particular, equation (\ref{F=0}) implies that:
\begin{align} &g_{(1,0)}^1((d_A^2(x_1))^2\otimes t^{-1})=0\notag\\ &(x_2^{(0)}g_{(0,1)}^1+x_1^{(0)}g_{(1,0)}^1)d_A^2(x_1)d_A^2(x_2)\otimes t^{-1}=0\notag\\ &(\frac{x_1^{(0)}}{2}-g_{(0,1)}^1x_1^{(0)})\Big((d_A^2(x_2))^2\otimes t^{-1}\Big)=0\notag
\end{align}
Hence, $g_{(0,1)}^1=0$ and also $g_{(0,1)}^1=\frac{x_1^{(0)}}{2}\neq 0$, a contradiction. Hence $F\neq0.$ 

Now we prove that $\phi_1^{\vee}(F)=0$. Firstly, note that:
\begin{align}
\phi^{\vee}_1(F)&=\phi^{\vee}_1(2((d_A^2(x_2))^2\otimes x_1^{(0)}x_1^{(1)})+\frac{1}{2}((d_A^2(x_2))^2\otimes (x_1^{(0)})^2t^{-1}))\notag\\
&=2x_1^{(0)}x_1^{(1)}\phi^{\vee}_1((d_A^2(x_2))^2\otimes1)+\frac{(x_1^{(0)})^2}{2}\phi^{\vee}_1((d_A^2(x_2))^2\otimes t^{-1})\notag\\
&=2x_1^{(0)}x_1^{(1)}(d_{A_1}^2(x_2^{(0)}))^2+(x_1^{(0)})^2d_{A_1}^2(x_2^{(0)})d_{A_1}^2(x_2^{(1)}).\notag
\end{align}
Set $\beta_1=(0,0,1,0)$ and observe that \begin{align}
F^1_{\beta_1}&=x_2^{(0)}d_{A_1}^2(x_1^{(0)})d_{A_1}^2(x_2^{(0)})+x_1^{(0)}(d_{A_1}^2(x_2^{(0)}))^2,\notag\\
F^2_{\beta_1}&=x_2^{(1)}d_{A_1}^2(x_1^{(0)})d_{A_1}^2(x_2^{(0)})+x_2^{(0)}d_{A_1}^2(x_1^{(1)})(d_{A_1}^2(x_2^{(0)}))^2+x_1^{(1)}d_{A_1}^2(x_2^{(0)})\notag\\&+x_1^{(0)}d_{A_1}^2(x_2^{(0)})d_{A_1}^2(x_2^{(1)}).\notag
\end{align}
Moreover,
\begin{align}x_1^{(1)}F^1_{\beta_1}+x_1^{(0)}F^2_{\beta_1}&=x_1^{(1)}\Big(x_2^{(0)}d_{A_1}^2(x_1^{(0)})d_{A_1}^2(x_2^{(0)})+x_1^{(0)}(d_{A_1}^2(x_2^{(0)}))^2\Big)\notag\\&+x_1^{(0)}\Big(x_2^{(1)}d_{A_1}^2(x_1^{(0)})d_{A_1}^2(x_2^{(0)})+x_2^{(0)}d_{A_1}^2(x_1^{(1)})d_{A_1}^2(x_2^{(0)})\notag\\&+x_1^{(1)}(d_{A_1}^2(x_2^{(0)}))^2+x_1^{(0)}d_{A_1}^2(x_2^{(0)})d_{A_1}^2(x_2^{(1)})\Big)\notag\\ &=\Big(x_1^{(1)}x_2^{(0)}+x_1^{(0)}x_2^{(1)}\Big)d_{A_1}^2(x_1^{(0)})d_{A_1}^2(x_2^{(0)})\notag\\ &+2x_1^{(0)}x_1^{(1)}(d_{A_1}^2(x_2^{(0)}))^2+\Big(x_1^{(0)}x_2^{(0)}\Big)d_{A_1}^2(x_1^{(1)})d_{A_1}^2(x_2^{(0)})\notag\\&+(x_1^{(0)})^2d_{A_1}^2(x_2^{(0)})d_{A_1}^2(x_2^{(1)})\notag\\&=2x_1^{(0)}x_1^{(1)}(d_{A_1}^2(x_2^{(0)}))^2+(x_1^{(0)})^2d_{A_1}^2(x_2^{(0)})d_{A_1}^2(x_2^{(1)}).\notag\end{align}
Therefore,
\begin{align}\phi_1^{\vee}(F)&=2x_1^{(0)}x_1^{(1)}(d_{A_1}^2(x_2^{(0)}))^2+(x_1^{(0)})^2d_{A_1}^2(x_2^{(0)})d_{A_1}^2(x_2^{(1)})\notag\\&=x_1^{(1)}F^1_{\beta_1}+x_1^{(0)}F^2_{\beta_1}=0.\notag\end{align}
\end{exam}

\section{A formula for the universal Hasse-Schmidt derivation and usual partial derivatives}

In the course of this investigation, we ran into a formula that relates Hasse-Schmidt derivations with usual partial derivatives. This relation seems to be well-known for derivatives of order 1 (see \cite[Section 5]{dFDo1} or \cite[Lemma 2.9]{BD} for an elementary proof of the following formula and see \cite[Proposition 2.3.11]{LNN} and \cite[Corollary 2.3.12 and 2.3.13]{LNN} for additional applications):

\begin{align}\label{fp}
\frac{\partial f^{(k)}}{\partial x^{(j)}_i}=d_{k-j}\left(\frac{\partial f}{\partial x_i}\right)\textnormal{ for all }0\leq j\leq k\leq n.
\end{align}

In this final section we provide a similar formula that relates the Hasse-Schmidt derivations with high-order partial derivatives.

For $\alpha=(\alpha^0_1,\ldots,\alpha^n_1,\ldots,\alpha^0_s,\ldots,\alpha^n_s)\in\N^{s(n+1)},$ set $$\hat{\alpha}:=(\sum_{j=0}^n\alpha_1^j,\sum_{j=0}^n\alpha_2^j,\ldots,\sum_{j=0}^n\alpha_s^j)\in\N^s.$$ For $i\in\{1,\ldots,s\}$, we define $\lambda_i:=\sum_{j=0}^nj\alpha_i^j$. Denote $\lambda_{\alpha}:=\sum_{i=1}^s\lambda_i$.
\begin{pro}
Let $l,n\in\N$. Given $\alpha\in\N^{s(n+1)}$, such that $0\leq\lambda_{\alpha}\leq l\leq n$. The the following identities hold as functions over the polynomial ring $\K[x_1,\ldots,x_s]$:\[\frac{\partial^{\alpha}}{\partial{x^{\alpha}}}\circ d_l=d_{l-\lambda_{\alpha}}\circ \frac{\partial^{\hat{\alpha}}}{\partial{x^{\hat{\alpha}}}}.\]
\end{pro}
\begin{proof}
Denote $\partial_{x^{\alpha}}=\frac{\partial^{\alpha}}{\partial_{x^{\alpha}}}.$ By the identity (\ref{fp}),
\begin{equation}\label{fpar}
\partial_{x^{(j)}_i}\circ d_k=d_{k-j}\circ \partial_{x_i},\mbox{ }0\leq j\leq k\leq n.
\end{equation} By definition, $\partial_{(x^{(j)}_i)^{\alpha_i^j}}\circ d_l=\partial_{x^{(j)}_i}\big(\partial_{x^{(j)}_i}\ldots\big(\partial_{x^{(j)}_i}\circ d_l\big)\big)$. Equation (\ref{fpar}) implies
\begin{align}\notag
\partial_{(x^{(j)}_i)^{\alpha_i^j}}\circ d_l&=\partial_{x^{(j)}_i}\big(\partial_{x^{(j)}_i}\ldots\partial_{x^{(j)}_i}\big(\partial_{x^{(j)}_i}\circ d_l\big)\big)\\\notag
&=\partial_{x^{(j)}_i}\big(\partial_{x^{(j)}_i}\ldots\partial_{x^{(j)}_i}\big(d_{l-j}(\partial_{x_i})\big)\big)\notag\\
&=\partial_{x^{(j)}_i}\big(\partial_{x^{(j)}_i}\ldots\partial_{x^{(j)}_i}\big(d_{l-2j}(\partial_{x_i}\circ\partial_{x_i})\big)\big).\notag
\end{align}
Iterating this process we obtain
\begin{align}\notag
\partial_{(x^{(j)}_i)^{\alpha_i^j}}\circ d_l&=d_{l-\alpha_i^j\cdot j}\circ\big(\partial_{x_i}\big(\partial_{x_i}\cdots(\partial_{x_i})\big)\big)\\
&=d_{l-\alpha_i^j\cdot j}\circ \partial_{x_i^{\alpha_i^j}}.\label{eq6}
\end{align}
Fix $i\in\{1\ldots,s\}$. Then, equation (\ref{eq6}) gives  
\begin{align}
\partial_{(x^{(0)}_i)^{\alpha_i^0}}\circ\cdots\circ\partial_{(x^{(n)}_i)^{\alpha_i^n}}\circ d_l&=\partial_{(x^{(0)}_i)^{\alpha_i^0}}\circ\cdots\circ\partial_{(x^{(n-1)}_i)^{\alpha_i^{n-1}}}\circ d_{l-\alpha_i^n\cdot n}\circ \partial_{x_i^{\alpha_i^n}}\notag\\
&=\partial_{(x^{(0)}_i)^{\alpha_i^0}}\circ\cdots\circ d_{l-(\alpha_i^nn)-(\alpha_i^{n-1}(n-1))}\circ\partial_{x^{\alpha_i^{n-1}}_i}\circ \partial_{x_i^{\alpha_i^n}}.\notag
\end{align}
Repeating this process results in
\begin{align}\notag
\partial_{(x^{(0)}_i)^{\alpha_i^j}}\circ\cdots\circ\partial_{(x^{(n)}_i)^{\alpha_i^j}}\circ d_l&= d_{l-\sum_{j=0}^n\alpha_i^j\cdot j}\circ \partial_{x_i^{\sum_{j=0}^n\alpha_i^j}}\\\notag
&=d_{l-\lambda_ i}\circ \partial_{x_i^{\sum_{j=0}^n\alpha_i^j}}.
\end{align}
The previous computations give place to the following identity. For $\alpha=(\alpha^0_1,\ldots,\alpha^n_1,\ldots,\alpha^0_s,\ldots,\alpha^n_s)\in\N^{s(n+1)}$ we have
\begin{align}\notag
\partial_{x^{\alpha}}\circ d_l&=\partial_{(x^{(0)}_1)^{\alpha_1^0}}\circ\cdots\circ\partial_{(x^{(n)}_s)^{\alpha_s^n}}\circ d_l\\\notag
&=\partial_{(x^{(0)}_1)^{\alpha_1^0}}\circ\cdots\circ\partial_{(x^{(n)}_{s-1})^{\alpha_{s-1}^{n}}}\circ d_{l-\lambda_ s}\circ \partial_{x_s^{\sum_{j=0}^n\alpha_s^j}}\\
&=\partial_{(x^{(0)}_1)^{\alpha_1^0}}\circ\cdots\circ\partial_{(x^{(n)}_{s-2})^{\alpha_{s-2}^{n}}}\circ d_{l-\lambda_ s-\lambda_{s-1}}\circ \partial_{x_{s-1}^{\sum_{j=0}^n\alpha_{s-1}^j}}\circ \partial_{x_s^{\sum_{j=0}^n\alpha_s^j}}.\notag
\end{align}
Applying repeatedly this step we conclude
\begin{align}\notag
\partial_{x^{\alpha}}\circ d_l&=d_{l-\sum_{i=1}^s\lambda_ i}\circ \partial_{x_1^{\sum_{j=0}^n\alpha_1^j}}\circ\cdots\circ\partial_{x_s^{\sum_{j=0}^n\alpha_s^j}}\\\notag
&=d_{l-\lambda}\circ \partial_{x^{\hat{\alpha}}}.
\end{align}
\end{proof}

\nocite{*}
\bibliographystyle{acm}
\bibliography{sample}

 \addcontentsline{toc}{section}{\numberline{References}}
\vspace{.5cm}
\noindent{\footnotesize \textsc {Paul Barajas, Universidad Aut\'onoma de Zacatecas, Paseo a La Bufa entronque Solidaridad s/n, CP 98000, Zacatecas, Mexico} \\
36178329@uaz.edu.mx, paulvbg@gmail.com}\\

\end{document}